\documentclass[11pt,a4paper,reqno]{amsart}
\usepackage{amsmath}
\usepackage{amsfonts}
\usepackage{amssymb}
\numberwithin{equation}{section}
\theoremstyle{plain}

\usepackage[utf8]{inputenc}
\usepackage[T1]{fontenc}
\usepackage{ae,aecompl}
\usepackage{amsmath}
\usepackage{amssymb}
\usepackage{amsthm}
\usepackage{array}
\usepackage{xy}
\usepackage{bm}
\usepackage{graphicx}
\date{}

\DeclareMathOperator{\conv}{conv}

\DeclareMathOperator{\vol}{vol}

\DeclareMathOperator{\inte}{int}

\newtheorem{twr}{Theorem}
\newtheorem{lem}[twr]{Lemma}

\newtheorem{coll}[twr]{Corollary}

\theoremstyle{remark}
\newtheorem{remark}[twr]{Remark}

\makeatletter
\renewcommand\@seccntformat[1]{\csname the#1\endcsname.\quad}

\makeatother

\begin{document}

\title{Stability result for the extremal Gr\"unbaum distance between convex bodies}

\author{Tomasz Kobos}
\thanks{The research of the author was supported by Polish National Science Centre grant 2014/15/N/ST1/02137}

\address{Faculty of Mathematics and Computer Science \\ Jagiellonian University \\ Lojasiewicza 6, 30-348 Krakow, Poland}

\email{Tomasz.Kobos@im.uj.edu.pl}


\begin{abstract}
In 1963 Gr\"unbaum introduced the following variation of the Banach-Mazur distance for arbitrary convex bodies $K, L \subset \mathbb{R}^n$: $d_G(K, L) = \inf \{ |r| \ : \ K' \subset L' \subset rK' \}$ with the infimum taken over all non-degenerate affine images $K'$ and $L'$ of $K$ and $L$ respectively. In 2004 Gordon, Litvak, Meyer and Pajor proved that the maximal possible distance is equal to $n$, confirming the conjecture of Gr\"unbaum. In 2011 Jim\'{e}nez and Nasz\'{o}di asked if the equality $d_G(K, L)=n$ implies that $K$ or $L$ is a simplex and they proved it under the additional assumption that one of the bodies is smooth or strictly convex. The aim of the paper is to give a stability result for a smooth case of the theorem of Jim\'{e}nez and Nasz\'{o}di. We prove that for each smooth convex body $L$ there exists $\varepsilon_0(L)>0$ such that if $d_G(K, L) \geq (1-\varepsilon)n$ for some $0 \leq \varepsilon \leq \varepsilon_0(L)$, then $d(K, S_n) \leq 1 + 40n^3r(\varepsilon)$, where $S_n$ is the simplex in $\mathbb{R}^n$, $r(\varepsilon)$ is a specific function of $\varepsilon$ depending on the modulus of the convexity of the polar body of $L$ and $d$ is the usual Banach-Mazur distance. As a consequence, we obtain that for arbitrary convex bodies $K, L \subset \mathbb{R}^n$ their Banach-Mazur distance is less than $n^2 - 2^{-22}n^{-7}$.
\end{abstract}

\maketitle

\section{Introduction}

Let $n \geq 2$ be an integer. We say that a set $K \subset \mathbb{R}^n$ is a \emph{convex body} if it is compact, convex and with non-empty interior. A convex body $K$ will be called \emph{centrally-symmetric} (or just \emph{symmetric}) if it has a center of symmetry. By $\partial K$ we shall denote the boundary of a convex body $K$. For $z \in \mathbb{R}^n$ let $K_z = K-z$ be a shift of $K$. If $0 \in \inte K$, then the \emph{polar body} $K^{\circ}$ of $K$ is defined as
$$K^{\circ} = \{ x \in \mathbb{R}^n \ : \ \langle x, y \rangle \leq 1 \text{ for every } y \in K \},$$
where $\langle \cdot, \cdot \rangle$ is the usual scalar product in $\mathbb{R}^n$.

If $K$ is a convex body such that $0 \in \inte K$, then we shall denote by $|| \cdot ||_K$ the \emph{gauge function} (or the \emph{Minkowski functional}) of $K$, that is
$$||x||_K = \inf \{ t>0 \ : \ x \in tK \}.$$
If $K$ is symmetric with respect to the origin, then $||x||_K$ is a norm in $\mathbb{R}^n$. On the other hand, the unit ball of an arbitrary norm in $\mathbb{R}^n$ is a convex body that is symmetric with respect to the origin. For an arbitrary convex body $K$ such that $0 \in \inte K$, the gauge function of $K$ possess almost the same properties as the norm: it is non-negative, definite, homogeneous for positive scalars and satisfies the triangle inequality. However, in general the equality $||x||_K=||-x||_K$ does not have to be true. We shall use general gauge functions extensively and it is important to keep that fact in mind. Moreover, it is clear that $||x||_K \leq 1$ if and only if $x \in K$ and $||x||_K = 1$ if and only if $x \in \partial K$.

A convex body $K \subset \mathbb{R}^n$ is called \emph{strictly convex} if the boundary of $K$ contains no non-degenerate line segment. $K$ is called \emph{smooth} if it has a unique supporting hyperplane at each boundary point. Well-known result states that if $0 \in \inte K$, then $K$ is smooth if and only if $K^{\circ}$ is strictly convex and vice versa. 

If a symmetric convex body $K$ is the unit ball of some normed space $X = (\mathbb{R}^n, || \cdot ||)$, then the polar body $K^{\circ}$ is the unit ball of the dual space $X^{\star}$.

We are ready now to introduce the reader to the central point of the paper, that is to measuring the distance between convex bodies. The \emph{Banach-Mazur distance} is a well-established notion of the geometry of Banach spaces. It was originally introduced by Banach as a multiplicative distance between normed spaces of the same dimension, but its definition is naturally extended to provide a distance between not necessarily symmetric convex bodies of the same dimension. If $X$ and $Y$ are normed spaces of the same dimension, the Banach-Mazur distance between $X$ and $Y$ is defined as
$$d(X, Y) = \inf ||T|| \cdot ||T^{-1}||,$$
with the infimum taken over all invertible operators $T:X \to Y$. For two (not necessarily symmetric) convex bodies $K, L \subset \mathbb{R}^n$, it is defined as
$$d(K, L) = \inf \{r: K + u \subset T(L+v) \subset r(K+u)\}.$$
with the infimum taken over all invertible operators $T: \mathbb{R}^n \to \mathbb{R}^n$ and $u, v \in \mathbb{R}^n$. It is not hard to check that both definitions agree on the symmetric convex bodies. It means that the distance between two normed spaces is equal to the distance of their unit balls.

The Banach–Mazur distance provides the natural framework for a comparison of the geometry of two convex bodies. It has numerous important applications in the fields of convex geometry, discrete geometry and local theory of Banach spaces. Banach-Mazur distance has already been extensively studied by several authors and many remarkable results were obtained. See \cite{tomczak} for a monograph in large part devoted to a detailed study of the Banach-Mazur distance from the viewpoint of functional analysis. Probably the most famous result concerning the distance between convex bodies is

\begin{twr}[John]
\label{twjohn}
Let $K \subset \mathbb{R}^n$ be a convex body. If $\mathcal{E}$ is a minimal volume ellipsoid containing $K$ and $c$ is center of $\mathcal{E}$ then $c + \frac{1}{n} (\mathcal{E}-c) \subset K$. If $K$ is centrally symmetric convex body then the constant $\frac{1}{n}$ can be replaced with $\frac{1}{\sqrt{n}}$.
\end{twr}

It turns out that the minimal volume ellipsoid containing $K$ is always unique. If $B_2^n$ is the Euclidean ball in $\mathbb{R}^n$, then it follows directly from the theorem of John that $d(K, B_2^n) \leq n$ for an arbitrary convex body $K \subset \mathbb{R}^n$ and $d(K, B_2^n) \leq \sqrt{n}$ for any centrally symmetric convex body $K \subset \mathbb{R}^n$. Thus, by the triangle inequality, the maximal possible Banach-Mazur distance between two convex bodies in $\mathbb{R}^n$ is bounded by $n^2$ and $n$ in the general and centrally-symmetric case respectively. Estimating the maximal possible Banach-Mazur distance more precisely than that turns out to be a very challenging problem. By a highly ingenious random construction of convex bodies (that led to a breakthrough in other open problems of asymptotic convex geometry) Gluskin in \cite{gluskin} was able to show that the bound for symmetric convex bodies is asymptotically optimal: there exists a constant $c>0$ such that for every integer $n \geq 1$ there are centrally symmetric convex bodies $K_n, L_n \subset \mathbb{R}^n$ such that $d(K_n, L_n) \geq c n$. In non-symmetric setting Rudelson in \cite{rudelson} was able to improve the asymptotic order of the upper bound, proving that there are positive constants $C$ and $\alpha$ such that for every $n \geq 1$ and arbitrary convex bodies $K, L \subset \mathbb{R}^n$ the inequality
\begin{equation}
\label{oszrudelson}
d(K, L) \leq C n^{\frac{4}{3}} (\log n)^{\alpha}
\end{equation}
is true. The best lower bound that is currently known for the asymmetric case is also linear.

In the non-asymptotic setting our state of knowledge about Banach-Mazur distance is in much more preeliminary stage. The only single case for which the maximal possible Banach-Mazur was exactly determined is the symmetric planar case. Stromquist in \cite{stromquist} has established that the maximal distance between symmetric convex bodies on the plane is equal to $\frac{3}{2}$ and it is achieved by the square and the regular hexagon. Lassak in \cite{lassakdiameter} has provided a short proof of the inequality $d(K, L) \leq 3$ for arbitrary convex bodies $K, L \subset \mathbb{R}^2$. The bound of $3$ was recently improved to $\tfrac{19-\sqrt{73}}4\approx 2.614$ by Brodiuk, Palko and Prymak (see \cite{brodiuk}). Besides that, nothing else is known about non-asymptotic bounds on the maximal Banach-Mazur distance.

Major contribution was made by Gordon, Litvak, Meyer and Pajor in \cite{gordon}, who were able to extend the idea of John to arbitrary convex bodies introducing the notion of \emph{John's decomposition in general case} (which will be recalled in the next section). As already said, there are some major difficulties in estimating precisely the maximal possible distance between two symmetric convex bodies or two arbitrary convex bodies. Perhaps surprisingly, it turns out that it is possible to determine exactly the maximal distance between a symmetric and an arbitrary convex body. A variation of the usual Banach-Mazur distance is called the \emph{Gr\"unbaum distance} and is defined as
$$d_G(K, L) = \inf \{|r| \ : \ K + u \subset T(L+v) \subset r(K+u)\},$$
where $K, L \subset \mathbb{R}^n$ are arbitrary convex bodies and the infimum is taken over all invertible operators $T: \mathbb{R}^n \to \mathbb{R}^n$ and $u, v \in \mathbb{R}^n$. Thus, instead of sandwiching the affine copy of $L$ between two positive homothets of $K$, we may use negative homothets as well. Clearly $d_G(K, L) \leq d(K, L)$ and $d_G(K, L) = d(K, L)$ if $K$ or $L$ is centrally-symmetric. Gr\"unbaum introduced this distance in \cite{grunbaum} and conjectured that the maximal possible distance is equal to $n$. It was confirmed more than $40$ years later by Gordon, Litvak, Meyer and Pajor, who gave a short proof based heavily on their decomposition theorem.

\begin{twr}[Gordon, Litvak, Meyer, Pajor \cite{gordon}]
\label{glmp}
Let $K, L \subset \mathbb{R}^n$ be arbitrary convex bodies. Then $d_G(K, L) \leq n$. Moreover, the equality is achieved by the simplex and an arbitrary symmetric convex body.
\end{twr}

A natural question arises: is it true that if $d_G(K, L)=n$, then is $K$ or $L$ a simplex? This is the conjecture of Jim\'{e}nez and Nasz\'{o}di stated in \cite{jimenez}. These authors have carefully followed all estimates in the proof of Theorem \ref{glmp} and have established a set of conditions that has to be satisfied in the case of equality $d_G(K, L)=n$. Based on them they proved:

\begin{twr}[Jim\'{e}nez and Nasz\'{o}di \cite{jimenez}]
\label{naszodi}
Let $K \subset \mathbb{R}^n$ be an arbitrary convex body and let $L \subset \mathbb{R}^n$ be a convex body that is smooth or strictly convex. If the equality $d_G(K, L) = n$ holds, then $K$ is the simplex.
\end{twr}
Theorem \ref{naszodi} is a broad generalization of a result of Leichtweiss who proved it for $L=B_2^n$ in \cite{leichtweiss} (it was also rediscovered later by Palmon in \cite{palmon}).

The aim of this paper is to take this line of research one step further. A lot of attention in the field of convex geometry is devoted to the stability of extremal properties of convex bodies. When some convex body is known to have some extremal property, the question of stability naturally arises. The simplex happens to be the extremal body for many properties and there are many stability results for the simplex that are known. See \cite{boroczky}, \cite{fleury}, \cite{groemer}, \cite{guo}, \cite{kiderlen}, \cite{schneiderstability2}, \cite{stephen} for some examples of stability results and \cite{schneiderstability1} for discussion on this topic. As Theorem \ref{naszodi} also characterizes the simplex by an extremal property, it is natural to ask about the stability version. This question was already raised by Schneider in \cite{schneiderstability1} for the case $L=B_2^n$.

The main result of the paper is a stability version of Theorem \ref{naszodi} in the smooth setting. The quality of the estimate depends on the quality of smoothness of a convex body $L$, which is expressed through \emph{modulus of convexity} of the polar body of $L$. For a not necessarily symmetric convex body $L$ such that $0 \in \inte L$, we define its \emph{modulus of convexity} as a function $\delta_L: [0, 1] \to [0, \infty)$ given by
$$\delta_L(t) = \inf \left \{ 1 - \left | \left | \frac{x+y}{2} \right | \right |_L \ : \ x, y \in L, \: ||x-y||_L \geq t \right \}.$$
Note that in contrast to the symmetric case -- in which $0$ is the center of symmetry -- there is no obvious choice for the origin and different shifts of the body $L$ will in general produce different moduli of convexity. In our result the choice of the origin (and in consequence the modulus of convexity of the polar body) is to certain degree free. In Section \ref{measuring} we shall investigate how the moduli of convexity of the body and that of its polar behaves under taking shifts. Let us state our main result. By $S_n$ we denote the regular simplex in $\mathbb{R}^n$.

\begin{twr}
\label{twglowne}
Let $L \subset \mathbb{R}^n$ be a smooth convex body such that $0 \in \inte L$ and $L \subset -nL$. Let $K \subset \mathbb{R}^n$ be an arbitrary convex body. There exists $\varepsilon_0(L)>0$ such that if $0 \leq \varepsilon \leq \varepsilon_0(L)$ and $d_G(K, L) \geq (1-\varepsilon)n$, then 
$$d(K, S_n) \leq 1 + 40n^3r,$$
where $0 \leq r=r(\varepsilon) < 1$ is the infimum of numbers satysfing the inequality
\begin{enumerate}
\item $r \cdot \delta_{L^{\circ}}\left ( \frac{r}{4n^3} \right ) \geq 4n^2 \varepsilon$ for a general smooth convex body $L$. Moreover, in this case we can take $\varepsilon_0(L)=\frac{\delta_{L^{\circ}}\left ( \frac{1}{80n^6} \right )}{80n^5}.$
\item $r \cdot \delta_{L^{\circ}}\left ( \frac{r}{2n^2} \right ) \geq 4n \varepsilon$ for a smooth convex body $L$ that is additionally centrally symmetric. Moreover, in this case we can take $\varepsilon_0(L)=\frac{\delta_{L^{\circ}}\left ( \frac{1}{40n^5} \right )}{80n^4}.$
\item $r = (16\varepsilon)^{\frac{1}{3}}$ for $L$ being an ellipsoid. Moreover, in this case we can take $\varepsilon_0(L)=\frac{1}{128000n^9}.$
\end{enumerate}
\end{twr}

Note that for $\varepsilon = \varepsilon_0(L)$ we get that $d(K, S_n) \leq 3$ in every case. We remark also that the smoothness of $L$ is equivalent to the fact that $\delta_{L^{\circ}}(t)>0$ for $t>0$. The number $r$ is therefore well-defined in each case and $r(\varepsilon) \to 0$ with $\varepsilon \to 0$. Roughly speaking, the rate of convergence of $r(\varepsilon)$ is more or less the same as of the inverse of a strictly increasing function $\varepsilon \cdot \delta_{L^{\circ}}(\varepsilon)$. Furthermore, the condition $L \subset -nL$ is restriction only on the way that the $L$ is shifted and not on the $L$ itself. Indeed, let us recall that a parameter
$$s(L)=\inf \{r>0: \: (L-z) \subset r(-L-z) \text{ for some } L \in K \}$$
is called the \emph{asymmetry constant} of $L$ and it is known that for each convex body
\begin{equation}
\label{stalaas}
1 \leq s(L) \leq n.
\end{equation}
Moreover, it is clear that the equality $s(L)=1$ holds if and only if $L$ is symmetric and it is known that the equality $s(L)=n$ implies that $L$ is a simplex. See \cite{grunbaum} for more details. Thus, each convex body can be placed in such a way that $0 \in \inte L$ and $L \subset -nL$.

We can see that Theorem \ref{twglowne} provides actually a series of stability results with the quality of the estimate being expressed with the help of the modulus of the convexity of the polar body $L^{\circ}$. In general case, the modulus of convexity of the dual body, or some other measure of the smoothness, is necessary to state the Theorem. To use this result in practice for a specific convex body, one needs to determine or estimate the modulus of convexity of the polar body. It is thus tempting to apply Theorem \ref{twglowne} for some specific convex body $L$ for which the modulus of convexity of the polar body is known, besides the case $L=B_2^n$ that is already covered in case $(3)$. We do so in the following Corollary. By $B_p^n$ we shall denote the unit ball of the $\ell_p$ norm in $\mathbb{R}^n$ for $1 < p < \infty$, that is
$$B_p^n = \{ x \in \mathbb{R}^n \ : |x_1|^p + |x_2|^p + \ldots + |x_n|^p \leq 1 \}.$$
By $p^*=\frac{p}{p-1}$ we denote the dual conjugate of $p$.

\begin{coll}
\label{wnioseklp}
If $2 \leq p < \infty$ and a convex body $K$ satisfies $d(K, B_p^n) \geq (1-\varepsilon)n$ for some $\varepsilon \leq \varepsilon_0$, then 
$$d(K, S_n) \leq 1 + C_0\varepsilon^{\frac{1}{3}},$$
where $\varepsilon_0 = \left ( 2^{19}(p-1)n^{14} \right )^{-1}$ and $C_0=320(p-1)^{\frac{1}{3}}n^{\frac{14}{3}}$.

If $1 < p < 2$ and a convex body $K$ satisfies $d(K, B_p^n)>(1-\varepsilon)n$ for some $\varepsilon \leq \varepsilon_0$, then 
$$d(K, S_n) \leq 1 + C_0 \varepsilon^{\frac{1}{q+1}},$$
where $q=p^*$, $\varepsilon_0= \left ( 80^{q+1} \cdot q \cdot n^{5q+4}  \right )^{-1}$ and  $C_0=160 \cdot n^{\frac{q}{q+1} + 4} \cdot \left ( q(q-1)^{q-1} \right )^{\frac{1}{q+1}}$.
\end{coll}

The modulus of convexity has been studied mostly in the symmetric case, which corresponds to the case of normed spaces (if $L=-L$, then $|| \cdot ||_L$ is a norm). Note that in this case, one can consider $\delta_L$ as a function defined on interval $[0,2]$. If $X$ is a normed space we shall write simply $\delta_X$, rather then $\delta_{B_X}$. There are numerous results, that develop connections between the properties of modulus of convexity as a real function and the geometry of the underlying normed space. For certain classes of normed spaces the modulus of convexity has been determined or estimated quite precisely. In general however, it can be quite arbitrary non-decreasing function (see \cite{hajek}). A normed space $X$, for which $\delta_X$ behaves similarly to $\delta_{\ell_p}$ is called \emph{$p$-uniformly convex space}. More precisely, if $p \in [2, \infty)$, then $X$ is \emph{$p$-uniformly convex with a constant $C$} if for each $t \in [0, 2]$ the inequality
$$\delta_X(t) \geq Ct^p$$
is true. The $\ell_p$ spaces are $2$-uniformly convex with a constant $\frac{p-1}{8}$ for $1 < p \leq 2$ and $p$-uniformly convex with a constant $(p2^p)^{-1}$ for $2 \leq p < \infty$ (see Lemma \ref{modullp}). We can easily generalize Corollary \ref{wnioseklp} to the duals of $p$-uniformly convex spaces.

\begin{coll}
\label{powertypeconvex}
Let $X$ be an $n$-dimensional normed space and $K \subset \mathbb{R}^n$ a convex body.  Assume that the dual space $X^{*}$ is $p$-uniformly convex with a constant $C$, for some $2 \leq p < \infty$. If $d(K, B_X) \geq (1-\varepsilon)n$ for some $\varepsilon \leq \varepsilon_0$, then
$$d(K, S_n) \leq 1 + C_0 \varepsilon^{\frac{1}{p+1}},$$
where $\varepsilon_0=C \left ( 2 \cdot 40^{p+1} \cdot n^{5p+4}  \right)^{-1}$ and $C_0=40n^3 \left ( \frac{2^{p+2}n^{2p+1}}{C}  \right )^{\frac{1}{p+1}}.$
\end{coll}

An equivalent result can be expressed directly through the smoothness of the space $X$, rather than through the convexity of the dual space $X^*$. For a normed space $X$ the function $\rho_X: [0, \infty) \to \mathbb{R}$ defined as
$$\rho_X(t) = \sup \left \{ \frac{||x + ty|| + ||x-ty|| - 2}{2} \ : \ ||x||=||y||=1  \right \}$$
is called the \emph{modulus of smoothness} of $X$. If for some $1 < p \leq 2$ and $C>0$, the modulus of smoothness of $X$ satisfies the inequality
$$\rho_X(t) \leq C t^p$$
for every $t \geq 0$, then $X$ is called \emph{p-uniformly smooth  with a constant C}. It is well-known that $X$ is $p$-uniformly smooth space if and only if the dual $X^{*}$ is $p^{*}$-uniformly convex (with perhaps different constants). Therefore, we have the following variation of the previous Corollary.

\begin{coll}
\label{powertypesmooth}
Let $X$ be an $n$-dimensional normed space and $K \subset \mathbb{R}^n$ a convex body. Assume that $X$ is $p$-uniformly smooth with a constant $C$, for some $1 < p \leq 2$. If $d(K, B_X) \geq (1-\varepsilon)n$ for some $\varepsilon \leq \varepsilon_0$, then
$$d(K, S_n) \leq 1 + C_0 \varepsilon^{\frac{1}{q+1}},$$
where $q=p^{*}$, $\varepsilon_0 =  \frac{(q-1)^{q-1}}{q^q} \left ( 2^{q+1} \cdot 40^{q+1} \cdot C^{q-1} \cdot n^{5q+4}  \right)^{-1}$ and $C_0=40n^3 \left ( \frac{C^{q-1}2^{2q+2}q^qn^{2q+1}}{(q-1)^{q-1}}  \right )^{\frac{1}{q+1}}.$
\end{coll}

The class of $p$-uniformly smooth spaces contains not only the $\ell_p$ spaces but also their subspaces, the Schatten trace classes (see \cite{tomczak2} and \cite{ball}) and $p$-convex, $q$-concave Banach lattices (see \cite{figiel}), among the others. We refer the reader to sections 1.e and 1.f in \cite{lindenstrauss} for some geometrical properties of $p$-uniformly convex and $p$-uniformly smooth spaces.

Another immediate consequence of Theorem \ref{twglowne} is the improvement of the upper bound of $n^2$ on the maximal distance between a pair of arbitrary convex bodies.

\begin{coll}
\label{wnioseksrednica}
Let $K, L$ be convex bodies in $\mathbb{R}^n$. Then $d(K, L) < n^2 - 2^{-22}n^{-7}$.
\end{coll}

Obviously the order of our estimate is significantly worse than of the asymptotic upper bound (\ref{oszrudelson}) established by Rudelson. However, as far as we know, it is currently the only known upper bound for the maximal possible Banach-Mazur distance in each specific dimension $n \geq 3$ better than $n^2$.

The paper is organized as follows. In Section \ref{johnpoz} we recall the \emph{John's position} of convex bodies that is the starting point of the proof of Theorem \ref{twglowne}. In Section \ref{measuring} we study how the moduli of convexity of a given convex body and its polar body are altered by a translation. Proofs of our main results are given in Section \ref{dowody}. In the last section of the paper we provide some final remarks.

\section{John's position of convex bodies}
\label{johnpoz}

For convex bodies $K, L \subset \mathbb{R}^n$, we say that $K$ is in a \emph{position of maximal volume} in $L$ if $K \subset L$ and for all affine images $K'$ of $K$ contained in $L$ we have $\vol(K') \leq \vol(K)$. A standard trick used for upper bounding the Banach-Mazur distance between two convex bodies, applied already by many authors, is to consider a position of maximal volume. Such an extremal assumption often allows to obtain some useful properties which can be used to upper-bound the distance between bodies. A result of Gordon, Litvak, Meyer and Pajor gives us precise conditions on so-called \emph{contact points} when we consider the position of maximal volume. If $K, L \subset \mathbb{R}^n$ are convex bodies we say that $K$ is in \emph{John's position} in $L$ if $K \subset L$ and
\begin{itemize}
\item $x = \sum_{i=1}^{m} a_i \langle x, u_i \rangle v_i$ for every $x \in \mathbb{R}^n$,
\item $0 =  \sum_{i=1}^{m} a_i u_i =  \sum_{i=1}^{m} a_i v_i,$
\end{itemize}
for some integer $m>0$, $\{u_i: 1 \leq i \leq m \} \subset \partial K \cap \partial L$, $\{v_i: 1 \leq i \leq m \} \subset \partial K^{\circ} \cap \partial L^{\circ}$ and positive $a_i$'s. It is easy to see that these conditions imply that $a_1+a_2 + \ldots + a_m = n$.

These two positions of convex bodies are related by the following result.
\begin{twr}[Gordon, Litvak, Meyer, Pajor \cite{gordon}]
\label{decomposition}
If $K, L \subset \mathbb{R}^n$ are convex bodies such that $K$ is in a position of maximal volume in $L$ and $0 \in \inte L$, then there exists $z \in \frac{n}{n+1}K$ such that $K- z$ is in John's position in $L-z$ with $m \leq n^2 + n$.
\end{twr}

\section{Measuring the convexity and smoothness with respect to different centers}
\label{measuring}

Note that in Theorem \ref{decomposition} it is in general necessary to perform a shift of convex bodies to place them in the John's position. Thus, in order to have control over how the smoothness of $L$ is measured, we need to know how the moduli $\delta_{L^{\circ}}(t)$ and $\delta_{(L_z)^{\circ}}(t)$ are related for a translation vector $z \in \inte L$. We start by showing how to relate the moduli $\delta_{L}(t)$ and $\delta_{L_z}(t)$. We will not use the following lemma later, but it is a natural starting point for our investigation.

\begin{lem}
\label{modulshift}
Let $L \subset \mathbb{R}^n$ be a convex body such that $0 \in \inte L$ and $L \subset -rL$ for some $r \geq 1$. Let $z \in (1-C)L$ for some $0 < C \leq 1$. Then
$$\delta_{L_z}\left ( t \right ) \geq \frac{\delta_{L}\left ( Ct \right )}{1+(1-C)r},$$
for every $0 \leq t \leq 1$.
\end{lem}

\begin{proof}
Assume that $||x||_{L_z}, ||y||_{L_z} \leq 1$ and $||x-y||_{L_z} \geq t$ for some $0 \leq t \leq 1$. By the definition of the gauge function it is clear that
\begin{equation}
\label{gauge}
||x||_{L_z}= \inf \{ t>0: ||x+tz||_L \leq t \}
\end{equation}
and in consequence equivalently we can write $||x+z||_L, ||y+z||_L \leq 1$ and $||x-y + tz||_L \geq t$. Observe that from the triangle inequality and the given condition $||z||_{L} \leq 1-C$ it follows that
$$||x-y||_L + t(1-C) \geq ||x-y||_L + t||z||_L \geq ||x-y + tz||_L \geq t$$
and in consequence
$$||x-y||_L \geq Ct.$$
If we denote $\tilde{x}=x+z$ and $\tilde{y}=y+z$, then $||\tilde{x}||_{L}, ||\tilde{y}||_{L} \leq 1$ and $||\tilde{x}-\tilde{y}||_L \geq Ct.$ From the definition of the modulus of convexity we have that 
$$||x+y+2z||_L=||\tilde{x} + \tilde{y}||_L \leq 2 - 2\delta_L \left ( Ct \right).$$
Let us take
$$a = \frac{2\delta_L \left ( Ct \right)}{1+(1-C)r}.$$
Then
$$||x+y+2z||_L \geq ||x+y+(2-a)z||_L - a||-z||_L \geq ||x+y+(2-a)z||_L - ar||z||_L$$
$$\geq ||x+y+(2-a)z||_L - (1-C)ra$$
so that
$$||x+y+(2-a)z||_L \leq 2 - \left (2 \delta_L \left ( Ct \right) - (1-C)ra \right )=2-a,$$
by the definition of $a$. It means that
$$||x+y||_{L_z} \leq 2 - a = 2 - \frac{2\delta_{L}\left ( Ct \right )}{1+(1-C)r}$$
and the conclusion follows.
\end{proof}

\begin{remark}
As explained in the previous section, in general the best possible $r$ is $s(L)$ -- the asymmetry constant of $L$. If $L$ is centrally symmetric with respect to $0$, then of course we can take $r=1$.
\end{remark}

Note that in general $(L_z)^{\circ}$ is not a translation of $L^{\circ}$, which raises the level of complexity. A lower bound on $\delta_{(L_z)^{\circ}}\left ( t \right )$ provided in the next lemma can be therefore expected to be somewhat weaker than previously.

\begin{lem}
Let $L \subset \mathbb{R}^n$ be a convex body such that $0 \in \inte L$ and $L \subset -rL$ for some $r \geq 1$. Let $z \in (1-C)L$ for some $0 < C \leq 1$. Then 
$$\delta_{(L_z)^{\circ}}\left ( t \right ) \geq \frac{\delta_{L^{\circ}}\left ( \frac{C^2t}{1-C+r} \right )}{(1-C)r+1},$$
for every $0 \leq t \leq 1$.
\end{lem}

\begin{proof}

Suppose that $||f||_{(L_z)^{\circ}}, ||g||_{(L_z)^{\circ}} \leq 1$ and $||f-g||_{(L_z)^{\circ}} \geq t$ for some $0 \leq t \leq 1$. By the definition of a polar body we have that
$$\sup_{x \in L} \langle x-z, f \rangle = \sup_{x \in L_z} \langle x, f \rangle = ||f||_{(L_z)^{\circ}} \leq 1,$$
so that $\langle x, f \rangle \leq 1 + \langle f, z \rangle$ for every $x \in L$. By taking $x=0$ we see that $1 + \langle f, z \rangle \geq 0$. Moreover 
\begin{equation}
\label{normaf}
||f||_{L^{\circ}} \leq 1 + \langle f, z \rangle=A.
\end{equation}
Similarly
\begin{equation}
\label{normag}
||g||_{L^{\circ}} \leq 1 + \langle g, z \rangle=B.
\end{equation}

Note that
$$\langle z, f \rangle \leq ||f||_{L^{\circ}} \cdot ||z||_L \leq \left( 1 + \langle f, z \rangle \right ) \cdot (1-C)$$
and hence
$$1-\frac{1}{\langle z, f \rangle+1}=\frac{\langle z, f \rangle}{\langle z, f \rangle+1} \leq 1-C.$$
It follows that $A=\langle f, z \rangle+1 \leq \frac{1}{C}$. Furthermore
$$||-z||_{L} \leq r||z||_{L} \leq (1-C)r$$
and therefore
$$\frac{1-A}{r(1-C)}=\left \langle \frac{-z}{r(1-C)}, f \right \rangle \leq ||f||_{L^{\circ}} \leq A,$$
by inequality (\ref{normaf}). This yields an estimate $A \geq \frac{1}{(1-C)r+1}$. In the same way we prove the similar inequalities for $B$ and in consequence
\begin{equation}
\label{oszacowania}
\frac{1}{(1-C)r+1} \leq A, B \leq \frac{1}{C}.
\end{equation}

From condition $||f-g||_{(L_z)^{\circ}} \geq t$ it follows that for some $u \in L$ we have
\begin{equation}
\label{skalarnyu}
\langle u-z, f-g \rangle \geq t.
\end{equation}
so that
$$\langle u, f-g \rangle \geq t + A-B.$$

Let $x=\langle u, f \rangle$ and $y= \langle u, g \rangle$. Then $y \leq x - t + B - A$. We can thus estimate
\begin{equation}
\label{sepproof}
Bx - Ay \geq Bx - A(x - t + B - A) =(A-B)(A-x) + At.
\end{equation}

We shall prove that functionals $\frac{f}{A}$ and $\frac{g}{B}$ are seperated in the norm $|| \cdot||_{L^{\circ}}$, that is

\begin{equation}
\label{oddzielenie}
\max \left \{ \left| \left |\frac{f}{A}-\frac{g}{B}\right| \right |_{L^{\circ}}, \left| \left |\frac{g}{B}-\frac{f}{A}\right| \right |_{L^{\circ}} \right \} \geq  \frac{C^2t}{1-C+r}.
\end{equation}

Suppose first that $A \geq B$. Then since $x=\langle u, f \rangle \leq ||f||_{L^{\circ}} \leq A$ (by condition (\ref{normaf})) and $B \leq \frac{1}{C}$ (by condition (\ref{oszacowania})), it follows that
$$(A-B)(A-x) + At\geq At \geq ABCt,$$
which combined with inequality (\ref{sepproof}) yields
$$\left \langle u, \frac{f}{A} - \frac{g}{B} \right \rangle =\frac{x}{A} - \frac{y}{B} \geq Ct,$$
and this gives us 

$$\left| \left |\frac{f}{A}-\frac{g}{B}\right| \right |_{L^{\circ}} \geq Ct \geq \frac{C^2t}{1-C+r}.$$

Now let us suppose that $B>A$. Note that it is sufficient to prove that
$$||g-f||_{(L_z)^{\circ}} \geq \frac{Ct}{1-C+r},$$
as then an analogous argument will yield the inequality
$$\left| \left |\frac{g}{B}-\frac{f}{A}\right| \right |_{L^{\circ}} \geq \frac{C^2t}{1-C+r}.$$
In this purpose, observe that by inequality (\ref{skalarnyu}) we have

$$\langle z-u, g-f \rangle \geq t.$$

Moreover, if we take $s=\frac{1-C+r}{C}$ then
$$||(1+s)z-u||_L \leq (1-C)(1+s) + r = s,$$
so that $||z-u||_{L_z} \leq s$ by the relation (\ref{gauge}). Thus
$$||g-f||_{(L_z)^{\circ}} \geq \frac{t}{||z-u||_{L_z}} \geq \frac{Ct}{1-C+r}$$
and the inequality (\ref{oddzielenie}) is proved.

Directly from the definition of the modulus of convexity for $L^{\circ}$ it now follows that
$$\left| \left |\frac{f}{2A}+\frac{g}{2B}\right| \right |_{L^{\circ}} \leq 1-\delta_{L^{\circ}}\left (  \frac{C^2t}{1-C+r} \right ),$$
which means that for every $x \in L$ the inequality
$$\left \langle x, \frac{f}{2A}+\frac{g}{2B} \right \rangle \leq  1-\delta_{L^{\circ}}\left (  \frac{C^2t}{1-C+r} \right )$$
holds.

From now on the roles of $f$ and $g$ will be symmetric and therefore without losing the generality we can assume that $B \geq A$. In particular $\frac{2A}{A+B} \leq 1$ and $\frac{B-A}{A+B} \geq 0$. In consequence, combining the inequality above with an estimate $||g||_{L^{\circ}} \leq B$ it follows that for every $x \in L$ we have
$$\left \langle x, \frac{f}{A+B}+\frac{g}{A+B} \right \rangle=\frac{2A}{A+B}  \left \langle x, \frac{f}{2A}+\frac{g}{2B} \right \rangle + \frac{B-A}{A+B} \left \langle x, \frac{g}{B} \right \rangle$$
$$\leq \frac{2A}{A+B} \left ( 1-\delta_{L^{\circ}}\left ( \frac{C^2t}{1-C+r} \right ) \right) + \frac{B-A}{A+B}=1 - \frac{2A}{A+B} \delta_{L^{\circ}}\left ( \frac{C^2t}{1-C+r} \right ).$$
Thus 
$$\left \langle x, f + g \right \rangle \leq A+B - 2A \delta_{L^{\circ}}\left ( \frac{C^2t}{1-C+r} \right )$$
and finally
$$\left \langle x-z, f + g \right \rangle \leq 2 - 2A \delta_{L^{\circ}}\left ( \frac{C^2t}{1-C+r} \right ) \leq 2 - \frac{2\delta_{L^{\circ}}\left ( \frac{C^2t}{1-C+r} \right )}{(1-C)r+1}$$
by the lower bound of (\ref{oszacowania}). It follows that
$$||f+g||_{(L_z)^{\circ}} \leq 2 - \frac{2\delta_{L^{\circ}}\left (\frac{C^2t}{1-C+r} \right )}{(1-C)r+1}$$
and the proof is finished.

\end{proof}

In the proof of Theorem \ref{twglowne} we will use the following two immediate corollaries.

\begin{coll}
\label{modulpolar}
Let $L \subset \mathbb{R}^n$ be a convex body such that $0 \in \inte L$ and $L \subset -nL$. Let $z \in \frac{n}{n+1} L$. Then
$$\delta_{(L_z)^{\circ}}\left ( t \right ) \geq \frac{\delta_{L^{\circ}}\left ( \frac{t}{4n^3} \right )}{2n},$$
for every $0 \leq t \leq 1$.
\end{coll}

\begin{coll}
\label{modulpolar2}
Let $L \subset \mathbb{R}^n$ be a centrally symmetric convex body with respect to the origin. Let $z \in \frac{n}{n+1} L$. Then
$$\delta_{(L_z)^{\circ}}\left ( t \right ) \geq \frac{\delta_{L^{\circ}}\left ( \frac{t}{2n^2} \right )}{2},$$
for every $0 \leq t \leq 1$.
\end{coll}


\section{Proofs of the main results}
\label{dowody}

In this section we prove our main results. We start by recalling well-known properties of the moduli of convexity of $\ell_p^n$ spaces.

\begin{lem}
\label{modullp}
Let $1 < p < \infty$. Then the modulus of convexity of the convex body $B_p^n$ satisfies 
\begin{enumerate}
\item $\delta_{B_p^n}(t) \geq \frac{p-1}{8}t^2$ for every $t \in [0, 1]$ if $p \leq 2$,
\item $\delta_{B_p^n}(t) \geq \frac{1}{p}\left ( \frac{t}{2} \right)^p$ for every $t \in [0, 1]$ if $p \geq 2$.
\end{enumerate}
\end{lem}

\begin{proof}

Part $(1)$ was established in \cite{meir}. Part $(2)$ follows easily from the closed formula: $\delta_{B_p^n}(t) = 1 - \left (1 - \left ( \frac{t}{2} \right)^p \right )^{\frac{1}{p}}$ for $p \geq 2$ (derived for example in \cite{hanner}) and an elementary calculus.
\end{proof}

We are ready to prove Theorem \ref{twglowne}.

\emph{Proof of Theorem \ref{twglowne}.}
Suppose that $\varepsilon \leq \varepsilon_0(L)$ as in the conditions of the theorem. Because of the definition of $\varepsilon_0(L)$ in every case we can assume that
\begin{equation}
\label{oszr}
r \leq \frac{1}{20n^3}.
\end{equation}
By applying a suitable affine transformation we can suppose that $K \subset L$ is in a position of maximum volume. By Theorem \ref{decomposition} there exists $z \in \frac{n}{n+1}K$ such that $K_z \subset L_z$ is in John's position. Then the proof of Theorem \ref{glmp} shows that $K_z \subset L_z \subset -nK_z$ (see \cite{gordon}).  To simplify the notation we shall write $K, L$ instead of $K_z, L_z$ and $\delta(t)$ instead of $\delta_{(L_z)^{\circ}}(t)$. We will keep the notation for the original modulus of convexity of the body $L^{\circ}$ (before the shift), that is $\delta_{L^{\circ}}(t)$. In the case $(3)$, that is for $L$ being ellipsoid, concentrity of the ellipsoids in the Theorem \ref{twjohn} of John allows us to assume that $z=0$.

Let us define
$$\varepsilon_1=\begin{cases}
\frac{\delta_{L^{\circ}}\left ( \frac{r}{4n^3} \right )}{2n} \text{ for $L$ arbitrary, }\\ 
\frac{\delta_{L^{\circ}}\left ( \frac{r}{2n^2} \right )}{2} \text{ for $L$ centrally symmetric,}\\
\frac{r^2}{8} \text{ for } L=B^n_2.
\end{cases}$$
In each case we have $\varepsilon_1 > \varepsilon$ by the definition of $r$. The proof will be conducted simultaneously for all possibilities $(1), (2), (3)$. Consider a set $X$ defined as
$$X=\{ x \in \mathbb{R}^n: ||x||_L=1, ||-x||_K \geq (1-\varepsilon_1)n \}$$ 
Clearly the set $X$ is non-empty, as otherwise $d_G(K, L) \leq (1-\varepsilon_1)n < (1-\varepsilon) n$. Let $x \in X$ be an arbitrary element. Denote $M = \max_{i=1, 2, \ldots, m} \langle x, v_i \rangle$. Then, by using the properties of John's decomposition listed in Section \ref{johnpoz}, we have
$$(1-\varepsilon_1)n \leq ||-x||_K = \left | \left |\sum_{i=1}^{m} a_i \left ( M-\langle x, v_i \rangle \right )u_i \right | \right|_K \leq \sum_{i=1}^{m} a_i(M-\langle x, v_i \rangle) = nM.$$

It follows that $1-\varepsilon_1 \leq M \leq 1$. Thus
\begin{equation}
\label{jedenel}
\langle x, v_i \rangle \geq 1 - \varepsilon_1 \text{ for some } 1 \leq i \leq m.
\end{equation}

Let 
$$F = \{ f \in \partial {L^{\circ}}: f(x)=1 \text{ for some } x \in X \}.$$
We claim that there exists $f_0 \in \conv F$ such that $||f_0||_{L^{\circ}} \leq r$. Assume the contrary. Then there exists $v \in \partial L$ such that $\langle f, v \rangle > r$ for every $f \in F$. We will prove that for $t>0$ defined as
$$t = \frac{\varepsilon}{(1-\varepsilon)nr}$$
we have the following: 
\begin{equation}
\label{implikacja}
\text{if } ||-x||_K=1, \text{ then } ||x+tv||_L > \frac{1}{(1-\varepsilon)n}.  
\end{equation}

In fact, for all $x \in \partial(-K)$ we have $||x||_L \geq \frac{1}{n}$. Suppose that $x \in \partial(-K)$ and $||x||_L \geq \frac{1}{(1-\varepsilon_1)n}$. Then, by the assumptions of the Theorem, we have $r \varepsilon_1 \geq 2n \varepsilon$ in cases $(1)$ and $(2)$. Hence in these possibilities we get
$$t = \frac{\varepsilon}{(1-\varepsilon)nr} \leq \frac{\varepsilon_1}{2(1-\varepsilon)n^2} < \frac{\varepsilon_1-\varepsilon}{(1-\varepsilon)(1-\varepsilon_1)n^2}.$$
Note that $L \subset -nK \subset -nL$ and therefore $||-v||_L \leq n$. Thus
\begin{equation}
\label{osztv}
t ||-v||_L < \frac{\varepsilon_1-\varepsilon}{(1-\varepsilon)(1-\varepsilon_1)n}.
\end{equation}
In case $(3)$ we have $r \varepsilon_1 \geq 2 \varepsilon$, but also $||-v||_L=1$, as $L=-L$ in this case. Therefore the inequality (\ref{osztv}) holds also in this situation, and it follows that 
$$||x+tv||_L \geq ||x||_L - t||-v||_L > \frac{1}{(1-\varepsilon_1)n} - \frac{\varepsilon_1-\varepsilon}{(1-\varepsilon)(1-\varepsilon_1)n}=\frac{1}{(1-\varepsilon)n}.$$
Now let us assume that $x \in \partial(-K)$ and $||x||_L < \frac{1}{(1-\varepsilon_1)n}$. Then $\tilde{x}=\frac{x}{||x||}_L \in X$. Let $f \in F$ be a supporting functional for $\tilde{x}$. Then
$$||x+tv||_L \geq \langle x + tv, f \rangle > ||x||_L + t r \geq \frac{1}{n} + t r = \frac{1}{n} + \frac{\varepsilon}{(1-\varepsilon)n} = \frac{1}{(1-\varepsilon)n},$$
and hence the implication (\ref{implikacja}) is established.

From implication (\ref{implikacja}) it follows that every boundary point of the convex body 
$$K'=-(1-\varepsilon)nK+(1-\varepsilon)ntv$$
lies strictly outside of $L$. Moreover, the intersection $K' \cap L$ is non-empty -- clearly $(1-\varepsilon)ntv \in K'$ but also $(1-\varepsilon)ntv \in L$ as $v \in L$ and $t < \frac{1}{(1-\varepsilon)n}$. We conclude that $L \subset \inte K'$. But this is a contradiction, since the inclusions
$$K \subset L \subset \inte K' = \inte \left( -(1-\varepsilon)nK+(1-\varepsilon)ntv \right )$$
imply that $d_G(K, L) < (1-\varepsilon)n$, which contradicts the assumptions of the theorem. Thus, for some $f_0 \in F$ we have $||f_0||_{L^{\circ}} \leq r$.

By the Carath\'{e}odory's Theorem we can write $f_0=\sum_{i=1}^{M} \gamma_i f_i$, where $f_i \in F$, $\gamma_i>0$ for $1 \leq i \leq M$, $\sum_{i=1}^{M} \gamma_i=1$ and $M \leq n+1$. Our next goal is to prove the equality $M=n+1$, but we will also obtain some other useful properties in the process. For $1 \leq i \leq M$ let $x_i \in X$ be such that $f_i(x_i)=1$. By property (\ref{jedenel}), for every $1 \leq i \leq M$ we can find $1 \leq j(i) \leq m$ such that
$$\langle x_i, v_{j(i)} \rangle \geq 1 - \varepsilon_1.$$
Note that possibly we could have $j(i)=j(k)$ for some $i \neq k$. Let $w_i \in K^{\circ}$ be such that $\langle -x_i, w_i \rangle \geq (1-\varepsilon_1)n$. 
Since $\frac{-w_i}{n} \in L^{\circ}$ we have that
$$2 - 2\delta\left ( \left | \left |v_{j(i)} + \frac{w_i}{n} \right | \right|_{L^{\circ}} \right ) \geq \left | \left |v_{j(i)} - \frac{w_i}{n} \right | \right |_{L^{\circ}} \geq \left \langle x_i, v_{j(i)} - \frac{w_i}{n}  \right \rangle \geq 2 - 2\varepsilon_1.$$
In the case of $L$ arbitrary from Corollary \ref{modulpolar} we get
$$\delta\left ( \left | \left |v_{j(i)} + \frac{w_i}{n} \right | \right |_{L^{\circ}} \right ) \leq \varepsilon_1 = \frac{\delta_{L^{\circ}}\left ( \frac{r}{4n^3} \right )}{2n} \leq \delta(r).$$
For $L$ centrally-symmetric, from Corollary \ref{modulpolar2} it follows that
$$\delta\left ( \left | \left |v_{j(i)} + \frac{w_i}{n} \right | \right |_{L^{\circ}} \right ) \leq \varepsilon_1 = \frac{\delta_{L^{\circ}}\left ( \frac{r}{2n^2} \right )}{2} \leq \delta(r).$$
Finally for $L$ being an ellipsoid, Lemma \ref{modullp} yields
$$\delta\left ( \left | \left |v_{j(i)} + \frac{w_i}{n} \right | \right |_{L^{\circ}} \right ) \leq \varepsilon_1 = \frac{r^2}{8} \leq \delta_{B_2^n}(r)=\delta(r).$$

Therefore, since the modulus of convexity is clearly monotonic, in each situation we have
$$ \left | \left |v_{j(i)} + \frac{w_i}{n} \right | \right |_{L^{\circ}} \leq r.$$

Actually, by the definition of the modulus of convexity it is also true that
$$2 - 2\delta\left ( \left | \left |-v_{j(i)} - \frac{w_i}{n} \right | \right|_{L^{\circ}} \right ) \geq \left | \left |v_{j(i)} - \frac{w_i}{n} \right | \right |_{L^{\circ}}$$
and for this reason
\begin{equation}
\label{vw}
\max \left \{\left | \left |v_{j(i)} + \frac{w_i}{n} \right | \right |_{L^{\circ}}, \left | \left |-v_{j(i)} - \frac{w_i}{n} \right | \right |_{L^{\circ}}   \right \} \leq r.
\end{equation}

It is straightforward to obtain by essentialy the same argument the inequality
\begin{equation}
\label{vf}
\max \left \{\left | \left |v_{j(i)} - f_i \right | \right |_{L^{\circ}}, \left | \left |f_i-v_{j(i)} \right | \right |_{L^{\circ}}   \right \} \leq r.
\end{equation}

Let $z_0 = \sum_{i=1}^{M} \gamma_i v_{j(i)}$. From the triangle inequality and relation (\ref{vf}) it easily follows that $||z_0||_{L^{\circ}} \leq 2 r$. Note that by inequality (\ref{vw}) for any $1 \leq i, k \leq M$ we have
\begin{equation}
\label{skaldol}
\langle u_{j(k)}, v_{j(i)} \rangle =  \left \langle u_{j(k)}, - \frac{w_{j(i)}}{n} \right \rangle + \left \langle u_{j(k)}, v_{j(i)} + \frac{w_{j(i)}}{n} \right \rangle \geq -\frac{1}{n} - r.
\end{equation}

Hence for a fixed $1 \leq k \leq M$
\begin{equation}
\label{gammaosz}
2 r \geq \langle u_{j(k)}, \sum_{i=1}^{m} \gamma_i v_{j(i)} \rangle \geq \gamma_{k} - \frac{1-\gamma_{k}}{n} - (1-\gamma_{k})r.
\end{equation}
Summing this over all $1 \leq k \leq M$, we get
$$(3m-1) r \geq 1 - \frac{M-1}{n},$$
which is equivalent to
$$M \geq \frac{n+1+n r}{1+3nr}.$$
As $r < \frac{1}{3n^2+n}$ by the estimate (\ref{oszr}) it can be now checked by hand that
$$\frac{n+1+n r}{1+3nr} > n$$
and the equality $M=n+1$ follows.

Inequality (\ref{skaldol}) provides a lower bound of the scalar product $\langle u_{j(k)}, v_{j(i)} \rangle$ for $k \neq i$. For the latter part of the proof we need also an upper bound of this quantity. We claim that for $1 \leq k, i \leq n+1$, $k \neq i$ we have
\begin{equation}
\label{skalgora}
\langle u_{j(k)}, v_{j(i)} \rangle < -\frac{1}{n} + 10n^2 r
\end{equation}
Indeed, from the inequality (\ref{gammaosz}) it follows directly that 
$$\gamma_k \leq \frac{1+3nr}{n+1+nr}$$
for $1 \leq k \leq n$.
Thus
$$\gamma_k = 1 - \sum_{i \neq k} \gamma_i \geq 1 - \frac{n+3n^2r}{n+1+n r}=\frac{1+nr - 3n^2r}{n+1+nr}.$$
In this way we have proved that
\begin{equation}
\label{gammaosz2}
\frac{1+nr- 3n^2r}{n+1+n r} \leq \gamma_k \leq \frac{1+3nr}{n+1+n r},
\end{equation}
for $1 \leq k \leq n+1.$ Now, from the fact that $||z_0||_{L^{\circ}} \leq 2 r$ and inequalities (\ref{skaldol}), (\ref{gammaosz2}), we conclude that
$$2 r \geq \left \langle u_{j(k)}, \sum_{l=1}^{n+1} \gamma_l v_{j(l)} \right \rangle = \gamma_i \left \langle u_{j(k)}, v_{j(i)} \right \rangle + \gamma_k + \sum_{l \neq i, l \neq i} \gamma_l \left \langle u_{j(k)}, v_{j(l)} \right \rangle$$
$$\geq \gamma_i \left \langle u_{j(k)}, v_{j(i)} \right \rangle + \frac{1+nr - 3n^2r}{n+1+n r} - \sum_{l \neq i, l \neq i} \gamma_l \left ( \frac{1}{n} + r \right )$$
$$\geq \gamma_i \left \langle u_{j(k)}, v_{j(i)} \right \rangle + \frac{1+nr - 3n^2r}{n+1+n r} - \frac{n-1+3n^2r-3nr}{n+1+nr} \left ( \frac{1}{n} + r \right )$$
$$= \gamma_i \left \langle u_{j(k)}, v_{j(i)} \right \rangle + \frac{n+n^2r - 3n^3r}{n^2+n+n^2 r} - \frac{n-1 - 4n r + 4n^2 r + 3n^3 r^2 - 3n^2 r^2}{n^2+n+n^2r}$$
$$=\gamma_i \left \langle u_{j(k)}, v_{j(i)} \right \rangle + \frac{1 + 4n r - 3n^2 r - 3n^3 r + 3n^2r^2 - 3n^3 r^2  }{n^2+n+n^2r}.$$
Hence
$$\left \langle u_{j(k)}, v_{j(i)} \right \rangle \leq \gamma_i^{-1} \cdot \frac{-1-2nr+5n^2r + 3n^3 r - n^2r^2 + 3n^3 r^2}{n^2+n+n^2r}.$$

By the estimate (\ref{oszr}) it is clear that the numerator of the fraction above is negative and thus
$$\left \langle u_{j(k)}, v_{j(i)} \right \rangle \leq \frac{n+1+n r}{1+3nr} \cdot \frac{-1-2nr+5n^2r + 3n^3 r - n^2r^2 + 3n^3 r^2}{n^2+n+n^2r}$$
$$=\frac{-1-2nr+5n^2r + 3n^3 r - n^2r^2 + 3n^3 r^2}{n+3n^2r}.$$
It follows from a direct computation that
$$\frac{-1-2nr+5n^2r + 3n^3 r - n^2r^2 + 3n^3 r^2}{n+3n^2r} < -\frac{1}{n} + 10n^2r$$
and inequality (\ref{skalgora}) is proved.

We are ready to move to the last part of the reasoning. Note that from the inequality (\ref{skalgora}) it clearly follows that $j(i) \neq j(k)$ for $i \neq k$ as $-\frac{1}{n} + 10n^2r<1$. Let $S = \conv \{ u_{j(1)}, u_{j(2)}, \ldots, u_{j(n+1)} \}$. Our aim is to show that the simplex $S$ is close in the Banach-Mazur distance to $K$. Since $S \subset K$, it is enough to prove that $K \subset (1 + 40n^3r)S$. Let $C = \conv \{ v_{j(1)}, v_{j(2)}, \ldots, v_{j(n+1)}  \}.$ The polar body $(C_{z_0})^{\circ}$ with respect to $z_0 \in \inte C$ is also a simplex (note that $z_0$ is in the interior of $C$ by conditions \ref{gammaosz2}). We start by showing that
\begin{equation}
\label{zawieranie1}
-nK \subset (1+3nr)(C_{z_0})^{\circ}.
\end{equation}
Let $x \in K$. Then, by inequality (\ref{vw}) and condition $||z_0||_{L^{\circ}} \leq 2r$ for any $1 \leq i \leq n+1$, we have
$$\left \langle -nx, v_{j(i)}-z_0 \right \rangle = \left \langle -nx, \frac{-w_{j(i)}}{n} \right \rangle + \left \langle -nx, v_{j(i)} + \frac{w_{j(i)}}{n} \right \rangle + \left \langle nx, z_0\right \rangle$$
$$=\left \langle x, w_{j(i)} \right \rangle + n \left \langle x, -v_{j(i)} - \frac{w_{j(i)}}{n} \right \rangle + n \left \langle x, z_0\right \rangle \leq 1 + 3n r.$$
This proves inclusion (\ref{zawieranie1}).

Now we claim that
\begin{equation}
\label{zawieranie2}
(C_{z_0})^{\circ} \subset -\beta n S,
\end{equation}
for $\beta = 1 + 30n^3r$.
To prove this inclusion we begin by showing that for every $1 \leq i, k \leq n+1$, $i \neq k$ we have
\begin{equation}
\label{sciana}
\left \langle -\beta n u_{j(i)}, v_{j(k)}-z_0 \right \rangle > 1
\end{equation}
In fact, since $K \subset -nK$ it is clear that $||-u_{j(i)}||_K \leq n$ and in consequence $\langle u_{j(i)}, z_0 \rangle \geq -2nr$. Thus, combining inequalities (\ref{skalgora}) and (\ref{oszr}) we get 
$$\left \langle -\beta n u_{j(i)}, v_{j(k)}-z_0 \right \rangle = - \beta n \left \langle u_{j(i)}, v_{j(k)} \right \rangle + \beta n \left \langle u_{j(i)}, z_0 \right \rangle $$
$$> - \beta n \left ( -\frac{1}{n} + 10n^2 r \right) - 2\beta n^2 r
\geq \beta(1-12n^3r) = (1+30n^3r)(1-12n^3r) \geq 1$$

Inequality (\ref{sciana}) is everything we actually need to prove inclusion (\ref{zawieranie2}). First we conclude that all facets of the simplex $- \beta n S$ are outside of $(C_{z_0})^{\circ}$. Indeed, by the definition of the polar body 
$$(C_{z_0})^{\circ} = \{ x \in \mathbb{R}^n \, : \, \langle x, v_{j(i)}-z_0 \rangle \leq 1 \text { for } 1 \leq i \leq n+1 \}.$$
If we consider the facet $\conv \{ -\beta n u_{j(1)}, - \beta n u_{j(2)}, \ldots, - \beta n u_{j(n)} \}$ of $-\beta n S$, then by inequality (\ref{sciana}) an arbitrary vertex $u_{j(i)}$ with $1 \leq i \leq n$ satisfies $\left \langle -\beta n u_{j(i)}, v_{j(n+1)}-z_0 \right \rangle > 1$. By taking the convex hull it follows that the whole facet lies outside of $(C_{z_0})^{\circ}$. By the same argument all remaining facets are also outside of $(C_{z_0})^{\circ}$. Therefore it is now sufficient to observe that $0$ is a common point of both simplices. Clearly $0 \in (C_{z_0})^{\circ}$. Suppose that $0 \not \in -\beta n S$. Then there exists a vector $v \in \mathbb{R}^n$ such that $\langle - \beta n u_{j(i)}, v \rangle < 0$ for every $1 \leq i \leq n+1$. But since 
$$0 \in \inte \conv\{v_{j(1)} - z_0, v_{j(2)} - z_0, \ldots, v_{j(n+1)} - z_0 \}$$
we can write $v$ as a linear combination with non-negative coefficients of some $n$ of the vectors $v_{j(1)} - z_0, v_{j(2)} - z_0, \ldots, v_{j(n+1)} - z_0$. Without losing the generality let us assume that $v = \sum_{i=1}^{n} c_i (v_i-z_0)$ for some $c_i \geq 0$. Then by inequality (\ref{sciana}) we have $\langle - \beta n u_{j(n+1)}, v \rangle \geq 0$. This is a contradiction which proves inclusion (\ref{zawieranie2}).

Finally, using inclusions (\ref{zawieranie1}) and (\ref{zawieranie2}) we get
$$K \subset - \frac{1+3nr}{n} (C_{z_0})^{\circ} \subset \beta(1+3nr)S = (1+30n^3r)(1+3nr) S \subset (1 + 40n^3r)S$$ 
and the conclusion follows.

\qed


We prove Corollaries \ref{wnioseklp}, \ref{powertypeconvex}, \ref{powertypesmooth} in one go.

\emph{Proof of Corollaries \ref{wnioseklp}, \ref{powertypeconvex}, \ref{powertypesmooth}}. Corollary \ref{powertypeconvex} follows directly from part $(2)$ of Theorem \ref{twglowne}. Corollary \ref{wnioseklp} follows directly from Corollary \ref{powertypeconvex}, combined with Lemma \ref{modullp} and the well-known fact $(\ell_{p})^* = \ell_{p*}$. It is therefore enough to prove Corollary \ref{powertypesmooth}. By  the Lindenstrauss formula (see \cite{lindenstrauss} page 61) we have the following equality for every $t \geq 0$.
$$\rho_X(t) = \sup \left \{ \frac{1}{2} t x - \delta_{X^{\star}}(x) \ : \ 0 \leq x \leq 2 \right \}.$$
Thus by our assumptions
\begin{equation}
\label{oszmodul}
Ct^p \geq \rho_X(t) \geq \frac{1}{2} x t - \delta_{X^{\star}}(x)
\end{equation}
for every $t \geq 0$ and $x \in [0, 2]$. In particular, for $x = \frac{r}{2n^2}$ and $t = \left ( \frac{x}{2Cp} \right )^{\frac{1}{p-1}} = \left ( \frac{r}{4n^2Cp} \right )^{\frac{1}{p-1}}$ by a direct computation we get that
$$\delta_{X^{*}}\left ( \frac{r}{2n^2} \right ) \geq r^q \cdot (4n^2q)^{-q} \cdot \left ( \frac{q-1}{C} \right)^{q-1}.$$
Hence, $r=\left ( \varepsilon \frac{C^{q-1}2^{2q+2}q^qn^{2q+1}}{(q-1)^{q-1}}  \right )^{\frac{1}{q+1}}$ satisfies conditions of part $(2)$ of Theorem \ref{twglowne}. To estimate the $\varepsilon_0$ we proceed in the same way, by using inequality \ref{oszmodul} for $x = \frac{1}{40n^5}$ and $t = \left ( \frac{x}{2Cp} \right )^{\frac{1}{p-1}}$. This finishes the proof.

\qed

The improvement of the upper bound of $n^2$ on the maximal Banach-Mazur distance is an application of part $(3)$ of Theorem \ref{twglowne}.

\emph{Proof of Corollary \ref{wnioseksrednica}.} Let $\varepsilon=2^{-22}n^{-9}$. Suppose on the contrary that $d(K, L) \geq (1-\varepsilon)n^2$. Since $d(K, B_2^n) \leq n$ and $d(L, B_2^n) \leq n$ we have $d(K, B_2^n) \geq (1-\varepsilon)n$ and $d(L, B_2^n) \geq (1-\varepsilon)n$. Since $\varepsilon \leq \varepsilon_0(B_2^n)$ by part $(3)$ of Theorem \ref{twglowne} yields
$$d(K, S_n) \leq 1 + 40n^{3} (16\varepsilon)^{\frac{1}{3}}=1+\frac{5}{8}.$$
Similarly $d(L, S_n) \leq 1+\frac{5}{8}$. Thus
$$d(K, L) \leq d(K, S_n) \cdot d(L, S_n) \leq \left (1+\frac{5}{8} \right)^2 < 3 < (1-\varepsilon)n^2,$$
which contradicts the assumption and the proof is finished.
\qed

\section{Concluding remarks}
\label{concluding}
Theorem \ref{naszodi} of Jim\'{e}nez and Nasz\'{o}di holds for each smooth or strictly convex body $L$. We were able to provide a stability version of their result only for the smooth case. It is natural to conjecture that it should be possible to establish a similar stability result also for the strictly convex case, in which the quality of the estimate would be expressed through the modulus of convexity of $L$. Note that in Lemma \ref{modulshift} we have already presented how to relate $\delta_L(t)$ with $\delta_{L_z}(t)$, where $z$ is a translation vector that is in general necessary in Theorem \ref{decomposition}.

It would be also interesting to know if the order of the estimation is optimal, at least in the case $L=B_2^n$. Note that for each symmetric convex body $L$ there exists a convex body $K$ such that $d(K, L) \geq (1-\varepsilon)n$ and $d(K, S_n) \geq 1 + \varepsilon$. Indeed, let $K$ be any convex body such $s(K)=(1-\varepsilon)n$ (the asymmetry constant of a convex body was defined in the first section of the paper). It is easy to see that for each pair of convex bodies $(A, B)$ we have $d(A, B) \geq \frac{s(A)}{s(B)}$. Thus $d(K, L) \geq \frac{s(K)}{s(L)}=(1-\varepsilon)n$ and also $d(K, S_n) \geq \frac{s(S_n)}{s(K)} = \frac{1}{1-\varepsilon} \geq 1+\varepsilon$. Improvement in the order of estimate would lead to a better numerical upper bound on the maximal possible Banach-Mazur distance as in Corollary \ref{wnioseksrednica}.

\end{document}